\documentclass[12pt]{amsart}

\usepackage{amsfonts,amsmath,amsthm}
\usepackage{latexsym}
\RequirePackage{graphicx}

\newtheorem{theorem}{Theorem}[section]

\newtheorem{lemma}[theorem]{Lemma}

\newtheorem{note}[theorem]{Note}

\theoremstyle{definition}

\newcommand{\Z}{\ensuremath{\mathbb{Z}}}
\newcommand{\R}{\ensuremath{\mathbb{R}}}

\def \< {\langle}
\def \> {\rangle}

\begin{document}

\title{Improved bounds on the supremum of autoconvolutions}
\author{M\'at\'e~Matolcsi, Carlos Vinuesa}
\address{M.M.: Alfr\'ed R\'enyi Institute of Mathematics, Realtanoda u 13-15, Budapest,
Hungary (also at BME Department of Analysis, Budapest,
H-1111, Egry J. u. 1).} \email{matomate@renyi.hu}

\address{C.V.: Universidad Aut\'onoma de Madrid, Ciudad Universitaria de Cantoblanco, Madrid, Spain.}
\email{c.vinuesa@uam.es}

\thanks{M. Matolcsi was supported by the ERC-AdG 228005}
\thanks{C. Vinuesa was supported by grants CCG08-UAM/ESP-3906 and DGICYT MTM2008-03880 (Spain).}

\date{\today}

\begin{abstract}
We give a slight improvement of the best known lower bound for the
supremum of autoconvolutions of nonnegative functions supported in
a compact interval. Also, by means of explicit examples we
disprove a long standing natural conjecture of Schinzel and
Schmidt concerning the extremal function for such
autoconvolutions.
\end{abstract}

\maketitle



\section{Introduction}

Consider the set $\mathcal F$ of all nonnegative real functions
$f$ with integral 1, supported on the interval $[-\frac{1}{4},
\frac{1}{4}]$. What is the minimal possible value for the supremum
of the autoconvolution $f\ast f$? This question (or equivalent
formulations of it) has been studied in several papers recently
\cite{sch, martin1, yu, martin}, and is motivated by its discrete
analogue, the study of the maximal possible cardinality of
$g$-Sidon sets (or $B_2[g]$ sets) in $\{1,\dots, N\}$. The
connection between $B_2[g]$ sets and autoconvolutions is described
(besides several additional results) in \cite{martin1, regicarlos,
ujcarlos}.

\medskip

If we define the autoconvolution of $f$ as
$$f\ast f(x) = \int f(t) f(x-t) \, dt,$$
we are interested in
$$S = \inf_{f\in \mathcal F} \|f\ast f\|_\infty$$
where the infimum is taken over all functions $f$ satisfying the above restrictions.

\medskip

This short note gives two contributions to the subject. On the one
hand, in Section \ref{improvedbound} we improve the best known
lower bound on $S$. This is achieved by following the ideas of Yu
\cite{yu}, and Martin \& O'Bryant \cite{martin}, and improving
them in two minor aspects. On the other hand, maybe more
interestingly, Section \ref{examples} provides counterexamples to
a long-standing natural conjecture of Schinzel and Schmidt
\cite{sch} concerning the extremal function for such
autoconvolutions. In some sense these examples open up the subject
considerably: at this point we do not have any natural conjectures
for the exact value of $S$ or any extremal functions where this
value could be attained. Upon numerical evidence we are inclined
to believe that $S\approx 1.5$, unless there exists some hidden
``magical'' number theoretical construction yielding a much smaller
value (the possibility of which is by no means excluded).

\medskip

In short, we will prove
$$1.2748 \le S \le 1.5098$$
which improves the best lower and upper bounds that were known for $S$.

\section{Notation}\label{notation}

Throughout the paper we will use the following notation (mostly
borrowed from \cite{martin}).

\medskip

Let $\mathcal F$ denote the set of nonnegative real functions $f$
supported in $[-1/4,1/4]$ such that $\int f(x) \, dx=1$. We define
the autoconvolution of $f$, $f\ast f(x) = \int f(t) f(x-t) \, dt$
and its autocorrelation, $f \circ f (x) = \int f(t) f(x+t) \, dt$.
We are interested in $S = \inf_{f\in \mathcal F} \|f\ast
f\|_\infty$. We remark here that the value of $S$ does not change
if one considers nonnegative step functions in $\mathcal F$ only.
This is proved in Theorem 1 in \cite{sch}. Therefore the reader
may assume that $f$ is square integrable whenever this is needed.

\medskip

We will need a parameter $0 < \delta \le 1/4$ and use the notation
$u=1/2+\delta$, and $\tilde g(\xi)=\frac{1}{u}\int_\R g(x)e^{-2\pi
i x\xi/u}dx$ for any function $g$. We will also use Fourier coefficients
of period 1, i.e. $\hat g(\xi)=\int_\R g(x)e^{-2\pi i x\xi} \, dx$
for any function $g$.

\medskip

We will need a nonnegative kernel function $K$ supported in
$[-\delta, \delta]$ with $\int K=1$. We will also need that
$\tilde K(j)\ge 0$ for every integer $j$. We are quite convinced
that the choice of $K$ in \cite{martin} is optimal, and we will
not change it (see equation \eqref{kx} below).

\section{An improved lower bound}\label{improvedbound}

We will follow the steps of \cite{martin} (which, in turn, is
based on \cite{yu}). We include here all the ingredients for
convenience (the proofs can be found in \cite{martin}).

\begin{lemma}\label{lemma1}{\rm [Lemmas 3.1, 3.2, 3.3, 3.4 in
\cite{martin}]}
With the notation $f, K, \delta, u$ as described above, we have

\medskip

\begin{equation}\label{eq1}\int (f\ast f (x)) K(x) \, dx \le \|f\ast f\|_\infty.
\end{equation}

\medskip

\begin{equation}\label{eq2}\int (f \circ f (x)) K(x) \, dx \le 1+\sqrt{\|f\ast f\|_\infty - 1}\sqrt{\|K\|_2^2-1}.
\end{equation}

\medskip

\begin{equation}\label{eq3}\int (f\ast f (x)+ f\circ f(x)) K(x) \, dx
=\frac{2}{u}+2u^2\sum_{j\ne 0}(\Re \tilde f(j))^2\tilde K(j).
\end{equation}

\medskip

Let $G$ be an even, real-valued, $u$-periodic function that takes
positive values on $[-1/4,1/4]$, and satisfies $\tilde G(0)=0$.
Then
\begin{equation}\label{eq4}
u^2\sum_{j\ne 0}(\Re \tilde f(j))^2\tilde K(j)\ge
\left
(\min_{0\le x\le 1/4} G(x) \right )^2
\cdot \left (
 \sum_{j:
\tilde G(j)\ne 0} \frac{\tilde G(j)^2}{\tilde K(j)} \right )^{-1}.
\end{equation}
\end{lemma}

\medskip

The paper \cite{martin} uses the parameter $\delta=0.13$ (thus
$u=0.63$), and the kernel function
\begin{equation}\label{kx}K(x)=\frac{1}{\delta}\beta \circ \beta
\left( \frac{x}{\delta} \right) \ \ \mathrm{where}
 \ \beta(x)=\frac{2/\pi}{\sqrt{1-4x^2}}  \ \ \left( -\frac{1}{2}<x<\frac{1}{2} \right)\end{equation} (note here
that $\|K\|_2^2<0.5747/\delta$). Finally, in equation \eqref{eq4}
they use one of Selberg's functions, $G(x)=G_{0.63, 22}(x)$ defined
in Lemma 2.3 of \cite{martin}. Combining the statements of Lemma
\ref{lemma1} above they obtain
\begin{eqnarray}\label{basic} \|f\ast f\|_\infty + 1 + \sqrt{\|f\ast f\|_\infty-1}\sqrt{\|K\|_2^2-1} \ge \\
\ge \frac{2}{u}+2\left (\min_{0\le x\le 1/4} G(x) \right )^2 \cdot \left ( \sum_{j: \tilde G(j)\ne 0} \frac{\tilde G(j)^2}{\tilde K(j)} \right )^{-1} \nonumber
\end{eqnarray}
and substituting the values and estimates they have for $u$,
$\tilde G(j)$, $\tilde K(j)$, $\min_{0\le x\le 1/4} G(x)$ and
$\|K\|_2^2$ the bound $\|f\ast f\|_\infty \ge 1.262$ follows.

\medskip

Our improvement of the lower bound on $\|f\ast f\|_\infty$ comes in two steps.
First, we find a better kernel function $G$ in equation
\eqref{basic}. This is indeed plausible because Selberg's
functions $G_{u,n}$ do not correspond to the specific choice of
$K$ in \cite{martin} in any way, therefore we can expect an
improvement by choosing $G$ so as to minimize the sum $\sum_{j:
\tilde G(j)\ne 0} \frac{\tilde G(j)^2}{\tilde K(j)}$, while
keeping $\min_{0\le x\le 1/4} G(x)\ge 1$.

\medskip

Next, we observe that if $\|f\ast f\|_\infty$ is small then the first Fourier
coefficient of $f$ must also be small in absolute value, and we
use this information to get a slight further improvement. We will
also indicate how the method could yield further improvements.

\begin{theorem}\label{main}
If $f: [\frac{-1}{4},\frac{1}{4}]\to \R_+$ is a nonnegative
function with $\int f=1$, then $\|f\ast f\|_\infty \ge 1.2748$.
\end{theorem}

\begin{proof}

Let $K(x)$ be defined by \eqref{kx}. As in \cite{martin} we make
use of the facts that $\|K\|_2^2< 0.5747/\delta$, and $\tilde
K(j)=\frac{1}{u}|J_0(\pi \delta j/u)|^2$ where $J_0$ is the Bessel
$J$-function of order 0.

\medskip

As described above, the main improvement comes from finding a
better kernel function $G$ in equation \eqref{basic}. Indeed, if
we set $G(x)=\sum_{j=1}^n a_j \cos (2\pi j x/u)$, then $\tilde
G(j)=\frac{a_{|j|}}{2}$ for $-n\le j \le n$ ($j\ne 0$), and thus
equation \eqref{basic} takes the form
\begin{eqnarray}\label{special}
\|f\ast f\|_\infty + 1 + \sqrt{\|f\ast f\|_\infty - 1} \sqrt{0.5747/\delta-1} \ge \\
\ge \frac{2}{u}+\frac{4}{u}\left (\min_{0\le x\le 1/4} G(x) \right )^2
\cdot \left ( \sum_{j=1}^{n} \frac{a_j^2}{|J_0(\pi \delta j/u)|^2} \right )^{-1}. \nonumber
\end{eqnarray}
For brevity of notation let us introduce the ``gain-parameter''
$a=\frac{4}{u}\left (\min_{0\le x\le 1/4} G(x) \right )^2\left (
 \sum_{j=1}^{n} \frac{a_j^2}{|J_0(\pi \delta
j/u)|^2} \right )^{-1}$. We note for the record that $a\approx
0.0342$ for the choices $\delta=0.13$ and $G(x)=G_{0.63,22}(x)$ in
\cite{martin}. For any fixed $\delta$ we are therefore led to the
problem of maximizing $a$ (while we may as well assume that
$\min_{0\le x\le 1/4} G(x)\ge 1$, as $G$ can be multiplied by any
constant without changing the gain $a$). This problem seems
hopeless to solve analytically, but one can perform a numerical
search using e.g. the ``Mathematica 6'' software. Having done so, we
obtained that for $\delta=0.138$ and $n=119$ there exists a
function $G(x)$ with the desired properties such that $a>0.0713$.
The coefficients $a_j$ of $G(x)$ are given in the Appendix.
Therefore, using this function $G(x)$ and $\delta=0.138$ in
equation \eqref{special} we obtain $S\ge 1.2743.$

\bigskip

{\bf Remark.} One can wonder how much further improvement could be
possible by choosing the optimal $\delta$ and the optimal $G(x)$
corresponding to it. The answer is that there is {\it very little}
room left for further improvement, the theoretical limit of the
argument being somewhere around $1.276$. To see this, let
$f_s(x)=\frac{1}{2}(f(x)+f(-x))$ denote the symmetrization of $f$,
let $\beta_{\delta}(x)=\frac{1}{\delta}\beta(\frac{x}{\delta})$
(where $\beta(x)$ is defined in \eqref{kx})  and reformulate
equation \eqref{eq3} as follows:
\begin{equation}\label{eq3uj}\int (f\ast f (x)+ f\circ f(x)) K(x) \, dx
= 2 \int (f_s \ast \beta_{\delta}(x))^2 dx = 2 \|f_s\ast
\beta_{\delta}\|_2^2.
\end{equation}
This equality is easy to see using Parseval and the fact that
$\tilde K(j)= u (\tilde \beta_{\delta}(j))^2$. Now, with
$\beta_{\delta}(x)$ being given, the best lower bound we can
possibly hope to obtain for the right hand side is $\inf_{f_s}
\|f_s\ast \beta_{\delta}\|_2^2$, where the infimum is taken over
all nonnegative, symmetric functions $f_s$ with integral 1. To
calculate this infimum, one can discretize the problem, i. e.
approximate $\beta_{\delta}(x)$ and $f_s(x)$ by step functions,
the heights of the steps of $f_s$ being parameters. Then one can
minimize the arising multivariate quadratic polynomial by
computer. Finally, we can use equations \eqref{eq1}, \eqref{eq2}
and \eqref{eq3uj} to obtain a lower bound for $\|f\ast
f\|_\infty$. We have done this\footnote{The authors are grateful
to M. N. Kolountzakis for pointing out that this minimization
problem can indeed be solved numerically due to convexity
arguments.} for several values of $\delta$ and it seems that best
lower bound is achieved for $\delta \approx 0.14$ where we obtain
$\|f\ast f\|_\infty \ge 1.276$. We remark that all this could be
done rigorously, but one needs to control the error arising from
the discretization, and the sheer documentation of it is simply
not worth the effort, in view of the minimal gain.

\bigskip

We can further improve the obtained result a little bit by
exploiting some information on the Fourier coefficients of $f$.
For this we need two easy lemmas.

\small
\begin{lemma}\label{lemma2b}
Using the notation $z_1=|\hat f(1)|$ and $k_1=\hat K(1)=\hat
K(-1)$, where $K$ is defined by equation \eqref{kx}, we have
\begin{eqnarray}&&\label{eq2b}\int (f\circ f (x)) K(x) dx \le \\
&&\le 1+2z_1^2k_1 +\sqrt{\|f\ast f\|_\infty-1-2z_1^4}\sqrt{\|K\|_2^2-1-2k_1^2}. \nonumber
\end{eqnarray}
\end{lemma}
\begin{proof}
This is an obvious modification of Lemma 3.2 in \cite{martin}.
Namely,
\begin{eqnarray*}
\int (f\circ f (x)) K(x) dx & = & \sum_{j\in\Z} (\widehat{f\circ
f}(j))\hat K(j) \\
& = & 1 + 2z_1^2k_1 + \sum_{j\ne 0,\pm 1} |\hat f(j)|^2\hat
K(j) \\
&\le&  1+2z_1^2k_1+ \sqrt{\sum_{j\ne 0,\pm 1} |\hat f(j)|^4}\sqrt{\sum_{j\ne
0,\pm 1} \hat K(j)^2}
\end{eqnarray*}

\begin{eqnarray*}
&=&1+2z_1^2k_1+\sqrt{\|f \ast f\|_2^2-1-2z_1^4}\sqrt{\|K\|_2^2-1-2k_1^2}\\
&\le& 1+2z_1^2k_1+\sqrt{\|f\ast f\|_\infty-1-2z_1^4}\sqrt{\|K\|_2^2-1-2k_1^2}.
\end{eqnarray*}
\end{proof}

The next observation is that $z_1$ must be quite small if $\|f\ast
f\|_\infty$ is small. This is established by an application of the
following general fact (the discrete version of which is contained
in \cite{green}).

\begin{lemma}\label{lemmah}
If $h$ is a nonnegative function with $\int h=1$, supported on the
interval $[-\frac{1}{2},\frac{1}{2}]$ and bounded above by $M$, then $|\hat h(1)|\le \frac{M}{\pi} \sin \frac{\pi}{M}$.
\end{lemma}
\begin{proof}
Observe first that
$$\hat h (1) = \int_{\R} h(x) e^{-2 \pi i x} \, dx = e^{-2 \pi i t} \int_{\R} h(x + t) e^{-2 \pi i x} \, dx$$
and with a suitable choice of $t$, the last integral, $\int_{\R}
h(x + t) e^{-2 \pi i x} \, dx$, becomes real and nonnegative.
Taking absolute values we get
$$|\hat h(1)| = \int_{\R} h(x + t)\cos (2 \pi x) dx.$$
The lemma becomes obvious now, because in order to maximize this integral, $h(x+t)$ needs to be concentrated on the largest values of the cosine function, so
$$|\hat h(1)| \le \int_{-\frac{1}{2M}}^{\frac{1}{2M}} M \cos (2 \pi x) \, dx = \frac{M}{\pi} \sin \frac{\pi}{M}.$$
\end{proof}

\normalsize
It is now easy to conclude the proof of Theorem \ref{main}. Assume
$\|f\ast f\|_\infty<1.2748$. By Lemma \ref{lemmah} we conclude that
$$|\hat f(1)|=\sqrt{|\widehat{f\ast f}(1)|}\le \sqrt{\frac{1.2748}{\pi} \sin
\frac{\pi}{1.2748}} < 0.50426.$$

However, using Lemma \ref{lemma2b} instead of equation \eqref{eq2}
we can replace equation \eqref{special} by
\begin{eqnarray}\label{modified}
\frac{2}{u}+a &\le& \|f\ast f\|_\infty+1+2z_1^2k_1 +\\
&+&\sqrt{\|f\ast f\|_\infty-1-2z_1^4}\sqrt{0.5747/\delta-1-2k_1^2} \nonumber
\end{eqnarray}
Substituting $\delta=0.138$, $k_1=|J_0(\pi \delta)|^2$ and
$a=0.0713$ we obtain a lower bound on $\|f\ast f\|_\infty$ as a function of $z_1$.
This function $l(z_1)$ is monotonically decreasing in the interval
$[0, 0.50426]$ therefore the smallest possible value for $\|f\ast f\|_\infty$ is
attained when we put $z_1=0.50426$. In that case we get
$\|f\ast f\|_\infty=1.27481$, which concludes the proof of the theorem.
\end{proof}

\bigskip

{\bf Remark.} In principle, the argument above could be improved
in several ways.

First, Lemma \ref{lemmah} does not exploit the fact that $h(x)$ is
an autoconvolution. It is possible that a much better upper bound
on $|\hat h(1)|$ can be given in terms of $M$ if we exploit that
$h=f\ast f$.

Second, for any value of $\delta \le 1/4$ and any suitable kernel
functions $K$ and $G$ we obtain a lower bound, $l(z_1)$, for $\|f\ast f\|_\infty$ as a
function of $z_1$. A bound $\|f\ast f\|_\infty \ge s_0$ will follow if $z_1$ does
not fall into the ``forbidden set'' $F=\{x: l(x) < s_0\}$. In the
argument above we put $s_0=1.2748$ and, with our specific choices
of $\delta$, $K$ and $G$, the forbidden set was the interval
$F = (0.504433,0.529849)$, and we could prove that $z_1$ must be
outside this set. However, when altering the choices of $\delta$,
$K$ and $G$ the forbidden set $F$ also changes. In principle it
could be possible that two such sets $F_1$ and $F_2$ are disjoint,
in which case the bound $\|f\ast f\|_\infty \ge s_0$ follows automatically.

Third, it is possible to pull out further Fourier coefficients
from the Parseval sum in Lemma \ref{lemma2b}, and analyze the
arising functions $l(z_1, z_2, \dots )$.

\section{Counterexamples}\label{examples}

Some papers in the literature conjectured that $S=\pi/2$, with the
extremal function being
$$f_0(x)=\frac{1}{\sqrt{2x+1/2}},\ x \in \left(-\frac{1}{4}, \frac{1}{4} \right).$$
Note that $\|f_0\ast f_0\|_\infty=\pi/2=1.57079\dots$ In particular, the
last remark of \cite{sch} seems to be the first instance where
$\pi/2$ is suggested as the extremal value, while the recent paper
\cite{martin} includes this conjecture explicitly as Conjecture 5.1.
In this section we disprove this conjecture by means of specific
examples. The down side of such examples, however, is that we do
not arrive at any reasonable new conjecture for the true value
of $S$ or the extremal function where it is attained.

\medskip

The results of this section are produced by computer search and we
do not consider them deep mathematical achievements. However, we
believe that they are important contributions to the subject,
mostly because they can save considerable time and effort in the
future to be devoted to the proof of a natural conjecture which is
in fact false. We also emphasize here that although we disprove
the conjectures made in \cite{sch} and in \cite{martin}, this does
not reduce the value of the main results of those papers in any
way.

\medskip

The counterexamples are produced by a computer search. This is
most conveniently carried out in the discretized version of the
problem. That is, we take an integer $n$ and consider only
nonnegative step functions which take constant values $a_j$ on the
intervals $[-\frac{1}{4}+\frac{j}{2n}, -\frac{1}{4}+\frac{j+1}{2n})$
for $j=0, 1, \dots, n-1$. This is equivalent to considering all the nonzero polynomials $P(x)=a_0+a_1x+\dots+
a_{n-1}x^{n-1}$ with nonnegative coefficients such that $\sum_{j=0}^{n-1}a_j=\sqrt{2n}$ and their squares
$P^2(x)=b_0+b_1x+\dots + b_{2n-2}x^{2n-2}$, and asking for the infimum of the
maximum of the $b_j$'s. Schinzel and Schmidt proved \cite{sch} that this value is $\ge S$ and its limit when
$n \to \infty$ is $S$.

\begin{note}
Our constant $S$  can also be defined as $S = \inf_{g \in \mathcal
G} \frac {\|f \ast f\|_{\infty}}{\|f\|_1^2}$ where
 $\mathcal G$ is the set of all nonnegative real functions
$g$, not identically $0$, supported on the interval
$[-\frac{1}{4}, \frac{1}{4}]$.

The same thing  happens in the discrete version. We can consider
the set $\mathcal P$ of all nonzero polynomials of degree $\le
n-1$ with nonnegative real coefficients $P(x) = a_0 + a_1x + \dots
+ a_{n-1}x^{n-1}$ and their squares $P^2(x) = b_0 + b_1x + \dots +
b_{2n-2}x^{2n-2}$ and ask for the value of
\begin{equation}\label{masikdef}2 n
\inf_{P \in \mathcal P} \frac{\max_j b_j}{\left( \sum_{j=0}^{n-1}
a_j \right)^2},\end{equation} and we will obtain the same value
$S$ as before.

Although our examples  will be ``normalized'' in order to fit the
first definitions (i.e. all integrals will be normalized to 1, and
all sums will be normalized to $\sqrt{2n}$), most of the
computations we have been carried out using these other ones
(which are more convenient and closer to the ones given by
Schinzel and Schmidt). This note also justifies the fact that it
is not a problem if we have an integral which is not exactly equal
to 1 or a sum of coefficients in a polynomial which is not exactly
equal to $\sqrt{2n}$ because of small numerical errors.
\end{note}

While we can only search for {\it local} minima numerically, using
the ``Mathematica 6'' software we have been able to find examples
of step functions with $\|f \ast f\|_{\infty} < 1.522$, much lower
than $\pi/2$. Subsequently, better examples were produced with the LOQO
solver (Student version for Linux and on the NEOS
server\footnote{We are grateful to Imre Barany and Robert J.
Vanderbei who helped us with a code for LOQO.}), reaching the
value $\|f \ast f\|_{\infty}=1.51237...$. The best example we are
currently aware of has been produced by an iterative algorithm
designed by M. N. Kolountzakis and the first author. The idea is as
follows: take any step function $f=(a_0, a_1, \dots, a_{n-1})$ as a
starting point, normalized so that $\sum a_j=\sqrt{2n}$. By means
of linear programming it is easy (and quick) to find the step
function $g_0=(b_0, b_1, \dots, b_{n-1})$ which maximizes $\sum
b_j$ while keeping $\|f\ast g_0\|_\infty \le \|f\ast f\|_\infty$
(obviously, $\sum b_j\ge \sqrt{2n}$ because the choice $g_0=f$ is
legitimate). We then re-normalize $g_0$ as
$g=\frac{\sqrt{2n}g_0}{\sum b_j}$. Then $\|f\ast g\|_\infty \le
\|f\ast f\|_\infty$ by construction. If the inequality is strict
then it is easy to see that for small $t>0$ the function
$h=(1-t)f+tg$ will be better than our original $f$, i. e. $\|h\ast
h\|_\infty \le \|f\ast f\|_\infty$. And we iterate this procedure
until a fix-point function is reached.

The best example produced by this method is included in the
Appendix, achieving the value $\|f\ast f\|_\infty = 1.50972...$.
Figure~\ref{fig:step} shows a plot of the autoconvolution of this function.

\begin{figure}[htb]
\center{\includegraphics[width=\textwidth]{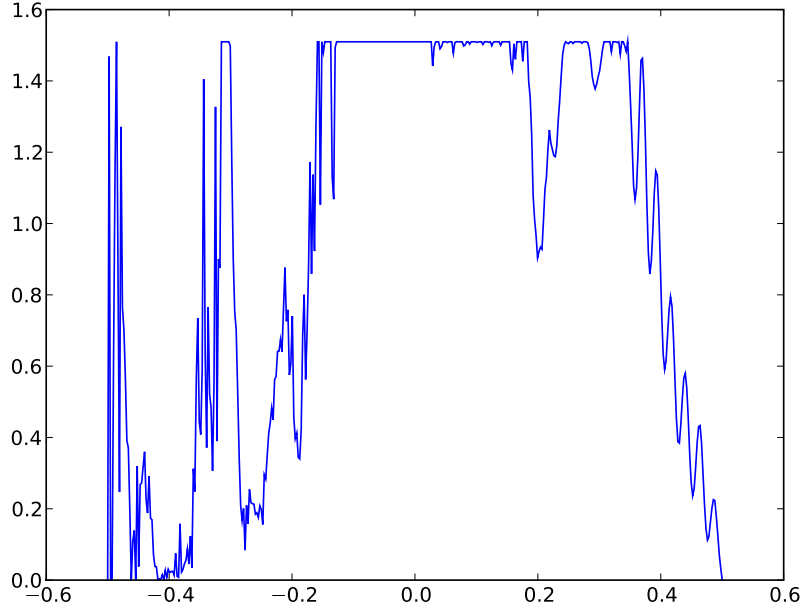}}
\caption{The autoconvolution of the best step function we are aware of, giving $\|f\ast f\|_\infty = 1.50972...$}
\label{fig:step}
\end{figure}

\medskip

Interestingly, it seems that the smallest value of $n$ for which a
counterexample exists is as low as $n = 10$, giving the value $1.56618...$
We include the coefficients of one of these polynomials here, as it is fairly easy
to check even by hand:
{\footnotesize
\begin{verbatim}
      0.41241661  0.45380115  0.51373388  0.6162143  0.90077119
      0.14003277  0.16228556  0.19989487  0.2837527  0.78923292
\end{verbatim}
}

\medskip

The down side of such examples is that it seems virtually
impossible to guess what the extremal function might be. We have
looked at the plot of many step functions $f$ with integral 1 and $\|f\ast
f\|_\infty <1.52$ and several different patterns seem to arise,
none of which corresponds to an easily identifiable function.
Looking at one particular pattern we have been able to produce an
analytic formula for a function $f$ which gives a value
for $\|f\ast f\|_\infty \approx 1.52799$, comfortably smaller than
$\pi/2$ but which is somewhat far from the minimal
value we have achieved with step functions. This function $f$ is
given as:
\begin{equation} \label{functioneq}
f(x) =
\left\{
\begin{array}{lll}
\dfrac {1.392887}{(0.00195 - 2x)^{1/3}}  \quad \textrm{ if } \ x \in (-1/4, 0) \\
\\
\dfrac{0.338537}{(0.500166 - 2x)^{0.65}} \quad \textrm{ if } \ x \in (0, 1/4) \\
\end{array}
\right.
\end{equation}

Figure~\ref{fig:functioneq} shows a plot of the autoconvolution of this function.

\begin{figure}[htb]
\center{\includegraphics[width=\textwidth]{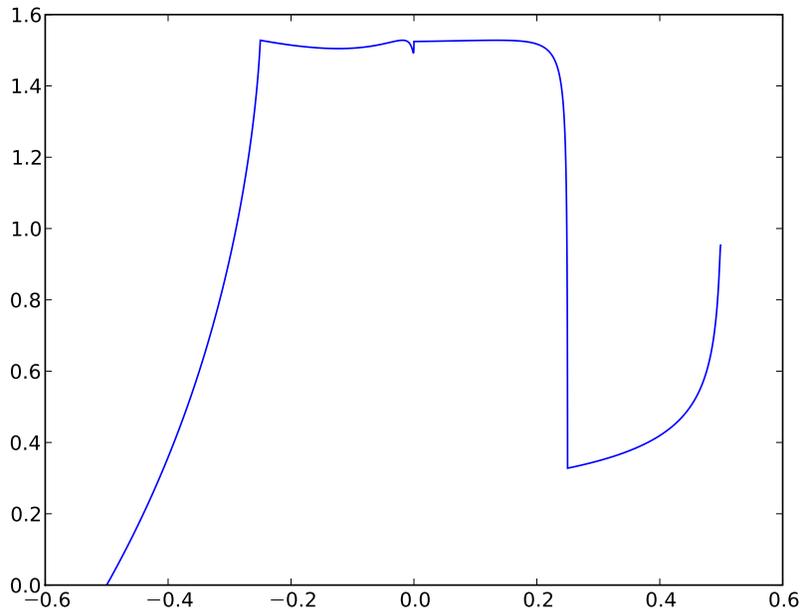}}
\caption{The autoconvolution of the function given by equation (\ref{functioneq}), giving $\|f\ast f\|_\infty \approx 1.52799$.}
\label{fig:functioneq}
\end{figure}

\bigskip

The paper \cite{martin} also states in Conjecture 2 that an
inequality of the form \begin{equation}\label{c}\|f\ast f\|_2^2\le
c \|f\ast f\|_\infty \|f\ast f\|_1 \end{equation} should be true
with the constant $c=\frac{\log 16}{\pi}$, and once again the
function $f_0$ above producing the extremal case. While we tend to
believe that such an inequality is indeed true with some constant
$c < 1$, we have been able to disprove this conjecture too, and
find examples where $c>\frac{\log 16}{\pi}$. We have not made
extensive efforts to maximize the value of $c$ in our numerical
search. In the Appendix we include one example of a step function
with $n = 20$ where $c = 0.88922... > \frac{\log 16}{\pi} =
0.88254...$

\medskip

We make a last remark here that could be of interest. It is
somewhat natural to believe that the minimal possible value of $\|
f \ast f \|_{\infty}$ does not change if we allow $f$ to take
negative values (but keeping $\int f=1$). However, this does not
seem to be the case. We have found examples of step functions $f$
for which $\| f \ast f \|_{\infty} = 1.45810...$, much lower than
the best value ($\| f \ast f \|_{\infty} =1.50972...$) we have for
nonnegative functions $f$. This example is also included in the
Appendix.

\section*{Akcnowledgements}
We thank Imre Ruzsa, Javier Cilleruelo, Mihail Kolountzakis and
Boris Bukh for many valuable suggestions and stimulating
discussions on the subject.

\section*{Appendix (online version only)}
Here we list the numerical values corresponding to the results of
the previous sections.

\medskip

For $\delta=0.138$ (and thus $u=0.638$) we define the kernel
function $G(x)$ used in Theorem \ref{main} as
$G(x)=\sum_{j=1}^{119} a_j \cos (2\pi j x/u)$, with the
coefficients $a_j$ given by the following list:

{\footnotesize
\begin{verbatim}
  2.16620392e+00  -1.87775750e+00   1.05828868e+00  -7.29790538e-01
  4.28008515e-01   2.17832838e-01  -2.70415201e-01   2.72834790e-02
 -1.91721888e-01   5.51862060e-02   3.21662512e-01  -1.64478392e-01
  3.95478603e-02  -2.05402785e-01  -1.33758316e-02   2.31873221e-01
 -4.37967118e-02   6.12456374e-02  -1.57361919e-01  -7.78036253e-02
  1.38714392e-01  -1.45201483e-04   9.16539824e-02  -8.34020840e-02
 -1.01919986e-01   5.94915025e-02  -1.19336618e-02   1.02155366e-01
 -1.45929982e-02  -7.95205457e-02   5.59733152e-03  -3.58987179e-02
  7.16132260e-02   4.15425065e-02  -4.89180454e-02   1.65425755e-03
 -6.48251747e-02   3.45951253e-02   5.32122058e-02  -1.28435276e-02
  1.48814403e-02  -6.49404547e-02  -6.01344770e-03   4.33784473e-02
 -2.53362778e-04   3.81674519e-02  -4.83816002e-02  -2.53878079e-02
  1.96933442e-02  -3.04861682e-03   4.79203471e-02  -2.00930265e-02
 -2.73895519e-02   3.30183589e-03  -1.67380508e-02   4.23917582e-02
  3.64690190e-03  -1.79916104e-02   7.31661649e-05  -2.99875575e-02
  2.71842526e-02   1.41806855e-02  -6.01781076e-03   5.86806100e-03
 -3.32350597e-02   9.23347466e-03   1.47071722e-02  -7.42858080e-04
  1.63414270e-02  -2.87265671e-02  -1.64287280e-03   8.02601605e-03
 -7.62613027e-04   2.18735533e-02  -1.78816282e-02  -6.58341101e-03
  2.67706547e-03  -6.25261247e-03   2.24942824e-02  -8.10756022e-03
 -5.68160823e-03   7.01871209e-05  -1.15294332e-02   1.83608944e-02
 -1.20567880e-03  -3.13147456e-03   1.39083675e-03  -1.49312478e-02
  1.32106694e-02   1.73474188e-03  -8.53469045e-04   4.03211203e-03
 -1.55352991e-02   8.74711543e-03   1.93998895e-03  -2.71357322e-05
  6.13179585e-03  -1.41983972e-02   5.84710551e-03   9.22578333e-04
 -2.16583469e-04   7.07919829e-03  -1.18488582e-02   4.39698322e-03
 -8.91346785e-05  -3.42086367e-04   6.46355636e-03  -8.87555371e-03
  3.56799654e-03  -4.97335419e-04  -8.04560326e-04   5.55076717e-03
 -7.13560569e-03   4.53679038e-03  -3.33261516e-03   2.35463427e-03
  2.04023789e-04  -1.27746711e-03   1.81247830e-04
\end{verbatim}
}

\medskip

The best nonnegative step function we are currently aware of,
reaching the value $\| f \ast f\|_{\infty} = 1.50972...$, is
attained at $n = 208$. The coefficients of its associate
polynomial (a polynomial of degree 207 whose coefficients sum up
to $\sqrt{416}$) are:

{\footnotesize
\begin{verbatim}
    1.21174638  0.          0.          0.25997048  0.47606812
    0.62295219  0.3296586   0.          0.29734381  0.
    0.          0.          0.          0.          0.
    0.          0.00846453  0.05731673  0.          0.13014906
    0.          0.08357863  0.05268549  0.06456956  0.06158231
    0.          0.          0.          0.          0.
    0.          0.          0.          0.          0.
    0.          0.          0.          0.          0.
    0.          0.          0.          0.          0.
    0.02396999  0.          0.          0.05846552  0.
    0.          0.          0.          0.          0.0026332
    0.0509835   0.          0.1283313   0.0904924   0.21232176
    0.24866151  0.09933512  0.01963586  0.01363895  0.32389841
    0.          0.          0.14467517  0.0129752   0.
    0.          0.16299837  0.38329665  0.11361262  0.32074656
    0.17344291  0.33181372  0.24357561  0.2577003   0.20567824
    0.13085743  0.17116496  0.14349025  0.07019695  0.
    0.          0.          0.          0.          0.
    0.          0.          0.          0.          0.
    0.          0.          0.          0.          0.
    0.          0.0131741   0.0342541   0.0427565   0.03045044
    0.07900079  0.07020678  0.08528342  0.09705597  0.0932896
    0.09360206  0.06227754  0.07943462  0.08176106  0.10667185
    0.10178412  0.11421821  0.07773213  0.11021377  0.12190377
    0.06572457  0.07494855  0.          0.          0.02140202
    0.          0.          0.0231478   0.00127997  0.
    0.04672881  0.03886266  0.11141784  0.00695668  0.0466224
    0.03543131  0.08803511  0.04165729  0.10785652  0.06747342
    0.18785215  0.31908323  0.3249705   0.09824861  0.23309878
    0.12428441  0.03200975  0.0933163   0.09527521  0.12202693
    0.13179059  0.09266878  0.02013746  0.16448047  0.20324945
    0.21810431  0.27321179  0.25242816  0.19993811  0.13683837
    0.13304836  0.08794214  0.12893672  0.16904485  0.22510883
    0.26079786  0.27367504  0.26271896  0.20457964  0.15073917
    0.11014028  0.09896     0.0926069   0.13269111  0.17329988
    0.20761774  0.21707182  0.18933169  0.14601258  0.08531506
    0.06187865  0.06100211  0.09064962  0.12781018  0.17038096
    0.185766    0.1734501   0.14667009  0.09569536  0.06092822
    0.03219067  0.0495587   0.09657756  0.16382398  0.22606693
    0.22230709  0.19833621  0.16155032  0.09330751  0.02838363
    0.02769322  0.03349924  0.09448887  0.20517242  0.22849741
    0.24175836  0.19700135  0.18168723
\end{verbatim}
}

\medskip

The best example of a step function disproving Conjecture 2 of
\cite{martin}, we are currently aware of, is attained for $n=20$
(note that we did not make extensive efforts to optimize this
example). {\footnotesize
\begin{verbatim}
  1.27283    0.54399    0.         0.         0.         0.
  0.         0.529367   0.410195   0.46111    0.439352   0.448675
  0.444699   0.446398   0.335601   0.322369   0.240811   0.202225
  0.138305   0.0886248
\end{verbatim}
} This function reaches the value  $c=0.88922...> \frac{\log
16}{\pi}$ in equation \eqref{c}.

\medskip

Finally, the best step function we are currently aware of (which takes some negative values!), reaching the value
$\| f \ast f \|_{\infty} = 1.45810...$, is attained at $n=150$. The coefficients of its associate polynomial are:
{\footnotesize
\begin{verbatim}
    0.7506545   0.4648332   0.59759775  0.46028561  0.36666088
    0.37773841  0.16162776  0.3303943   0.15905831  0.08878588
    0.16284952 -0.09198076  0.05755583 -0.00690908 -0.08627636
   -0.17180424 -0.14778207  0.13121791  0.05268415  0.20694965
    0.25287625  0.2071192  -0.13591836  0.05354584 -0.03558645
    0.15699341 -0.06508942 -0.01435246  0.02291645  0.18877783
   -0.02751401  0.09592962  0.06666674  0.1807308   0.15543041
    0.02639022  0.01843893  0.04896963  0.0303207   0.05119754
    0.24099308  0.2244329   0.23689694  0.08980581  0.25272138
    0.26725296  0.12786816  0.16265063  0.20542404  0.06826679
    0.16905985 -0.11230055  0.26179213 -0.412312   -0.28820566
   -0.7619902  -0.78933468  0.07066217  0.05785475  0.07163788
    0.09949514  0.0659708   0.05370837  0.08441868  0.10157278
    0.07317574  0.0521853   0.08980666  0.13113512  0.05943309
    0.07517572  0.12460218  0.14885796  0.09071907  0.13017884
    0.13185969  0.15196722  0.07848544  0.14924624  0.16053609
    0.17735544  0.14470971  0.17275872  0.16058981  0.22807136
    0.20728811  0.10876597  0.21471959  0.25136905  0.15147268
    0.06366331  0.05917714  0.05995267  0.35288009  0.3224057
    0.32988077  0.41806458  0.22880318  0.2080819   0.18504847
    0.27116284  0.16066195  0.02547032  0.26150045 -0.00634039
    0.09471136 -0.00407705  0.04759596 -0.07549638 -0.30815721
   -0.00878173  0.08964445  0.23265916  0.37008611  0.18283593
    0.00240797  0.063899    0.02892268  0.10802879  0.15672677
   -0.11335258  0.10549109  0.1571762   0.13290998 -0.01251118
    0.15487122  0.15770952  0.33037764  0.03888211  0.08105707
    0.00799348  0.00375632 -0.02392944  0.15019215  0.21615677
    0.17854093  0.04104506  0.12700956  0.23964236  0.05613369
    0.14857745  0.07375734  0.02816608  0.16226977  0.01757525
   -0.23848002  0.05705152  0.29372066  0.56730329  1.105205
\end{verbatim}
}

\end{document}